\documentclass[11pt,oneside]{amsart}
\usepackage{amssymb, amscd, amsmath, amsthm, latexsym, hyperref,ams refs}
\usepackage{pb-diagram, fancyhdr, graphicx, psfrag, enumerate}
\usepackage[text={6.5in,9.0in},centering,letterpaper,dvips]{geometry}
\usepackage{enumitem}\setlist[enumerate,1]{font=\upshape, itemsep=1ex}\setlist[itemize,1]{font=\upshape, itemsep=1ex}
\usepackage{pinlabel}
\usepackage{xcolor}

\usepackage{setspace}
\def\T{{\mathbb T}}
\def\Z{{\mathbb Z}}
\def\F{{\mathbb F}}
\def\Q{{\mathbb Q}}
\def\R{{\mathbb R}}

\def\P{{\mathbb P}}

\def\co{\colon\thinspace}

\def\calk{\mathcal{K}}
\def\cals{\mathcal{S}}

\def\calc{\mathcal{C}}

\def\calk{\mathcal{K}}
\def\calj{\mathcal{J}}
\def\calt{\mathcal{T}}

\def\cs{\mathbin{\#}}
\def\t2k{$(2,2k\! +\!1)$}
\DeclareMathOperator{\cfk}{\rm CFK}

\newcommand{\spinc}{\ifmmode{{\mathfrak s}}\else{${\mathfrak s}$\ }\fi}
\newcommand{\spinct}{\ifmmode{{\mathfrak t}}\else{${\mathfrak t}$\ }\fi}
\newcommand{\spincr}{\ifmmode{{\mathfrak r}}\else{${\mathfrak r}$\ }\fi}

\newtheorem{theorem}{Theorem}[section] 
\newtheorem{lemma}[theorem]{Lemma}
\newtheorem{corollary}[theorem]{Corollary}
\newtheorem{proposition}[theorem]{Proposition}

\theoremstyle{definition}

\newtheorem{definition}[theorem]{Definition}
\newtheorem{example}[theorem]{Example}

\begin{document}
\title[A metric on the projective knot concordance group]{Cobordism distance on the projective space of the  knot concordance group}

\author{Charles Livingston}
\thanks{This work was supported by a grant from the National Science Foundation, NSF-DMS-1505586.   }
\address{Charles Livingston: Department of Mathematics, Indiana University, Bloomington, IN 47405}\email{livingst@indiana.edu}

%\date{\today}

%%%%%%%ABSTRACT%%%%%%%%%%%%%%

\begin{abstract}   We use the cobordism distance on the smooth knot concordance group $\calc$ to measure how close two knots are to being linearly dependent.  Our measure, $\Delta(\calk, \calj)$, is built by minimizing the cobordism distance between all pairs of knots, $\calk'$ and $\calj'$, in cyclic subgroups containing $\calk$ and $\calj$. When made precise, this leads to the definition of the projective space of the concordance group, $\P(\calc)$, upon which $\Delta$ defines an integer-valued metric $\Delta$. We explore basic properties of $\P(\calc)$ by  using torus knots $T_{2,2k+1}$.  Twist knots are used to demonstrate that the natural simplicial complex  $\overline{ (\P(\calc), \Delta)}$ associated to the metric space  $\P(\calc)$ is infinite dimensional.
\end{abstract}

\maketitle

%%%%%%%Section%%%%%%%%%%%%%%

\section{Introduction}  
We let $\calc$ denote the smooth  concordance group of knots in $S^3$.   There are two fundamental measures of relationship between elements in $\calc$: one is algebraic, linear independence; the other is geometric, the cobordism distance, a metric defined by $d(\calk, \calj) = g_4( \calk \cs -\calj)$, where $g_4$ denotes the four-genus.
Here we combine these approaches to build a  measure of how  close knots are to being dependent in $\calc$.  Roughly stated, the distance between a pair of knots  is defined as the  minimum cobordism distance between all possible multiples and divisors of the knots.  When made formal, this leads to the definition of the projective space of the concordance group  $\P(\calc)$   and a metric  $\Delta$  built from the cobordism distance.

To make this somewhat more precise, for any vector space $V$ over a field $\F$, there is relation on $V^o = V \setminus 0$ defined by $v \sim w$ if there exists an  $\alpha \in \F$,    $\beta \in \F$,  and $z \in V$ such that   $\alpha z = v$ and $\beta z  = w$.   The set of equivalence classes is the projective space $\P(V)$.  We can define a similar relationship on an abelian group $G$, replacing the field $\F$ with the integers, $\Z$.  In this case, the relation might not be transitive, but it does generate an equivalence relation.  The equivalence classes form the  associated projective space of the  group,   $\P(G)$.

In the case of $\P(\calc)$,  we can define a metric based on the cobordism distance:  again  roughly stated, for classes $[\calk] \in \P(\calc)$ and $[\calj] \in \P(\calc)$, the distance $\Delta([\calk], [\calj])$  is defined by minimizing $d(\calk' , \calj')$ over all representatives of the classes $[\calk]$ and $[\calj]$.   To define  this formally, there are some technical modifications required to ensure that the triangle inequality holds.  
Our goals include the following.

\begin{itemize}
\item  Define $\P(\calc)$ and study its basic properties.  Letting $\calt \subset \calc$ denote the torsion subgroup, we show there is a natural bijection between $\P(\calc)$  and the disjoint union $\P(\calt) \sqcup \P( \calc / \calt)$.  We show that $\P(\calt)$ is either trivial or in bijective correspondence with  $\Z_2^\infty \setminus 0$ depending on whether or not $\calc$ contains odd torsion.  There is a natural bijection between $\P(\calc /\calt)$ and the infinite rational projective space $\Q\P^\infty$.

\item Define the metric $\Delta \co \P(\calc) \times \P(\calc) \to \Z$.  This gives $\P(\calc)$ the structure of a graph.  %Show that $\P(\calc)$ contains a tree for which all valences are infinite.  Also, $\P(\calc)$ contains arbitrarily long geodesics.

\item  Provide basic examples by studying the image of the  set of $(2,2k +1)$--torus knots in $\P(\calc)$.  

\item Define the associated {\it Vietoris-Rips} simplicial complex $\big| (\P(\calc), \Delta)\big|$ and use twist knots to prove it is infinite dimensional.

\item  Identify basic problems related to $(\P(\calc), \Delta)$.

\end{itemize}

\medskip
\noindent{\bf Background.}
We will work in the smooth category, but observe that our work here carries over to the topological locally-flat category.  It is possible to apply such tools as Upsilon invariant of Heegaard Floer theory~\cite{MR3667589} to achieve some finer results and in particular to  explore the subspace of $\P(\calc)$ generated by topologically slice knots.   For our basic examples, the signature function, which applies equally in both categories is more effective; see~\cite{MR3784267} for a comparison of the bounds on the four-genus of differences of torus knots that are determined by Upsilon and by the signature function.

In the 1960s, the combined work of Fox-Milnor~\cite{MR0211392}, Murasugi~\cite{MR0171275},  Milnor~\cite{MR0242163},   Levine~\cite{MR0253348}, and Tristam~\cite{MR0248854} demonstrated that as an abstract group, $\calc     \cong \Z^\infty \oplus \Z_2^\infty \oplus G$ for some countable abelian group $G$.   Nothing more is known today.

Despite that lack of progress in understanding $\calc$ from a purely algebraic standpoint, from a topological perspective tremendous strides have been made.  For instance, there are natural homomorphisms  of $\calc$ onto the topological concordance group and  onto the higher dimension knot concordance groups.  The   kernels of these maps are now known to contain infinite free summands  and to contain infinite 2--torsion; a few references include~\cite{MR620010, MR1734416, MR1334309}.    The primary examples used in proving the basic results about $\calc$ have been built from two-bridge knots and $(2,2k+1)$--torus knots.   Continuing research concerns further understanding the image of  classes of knots, such as two-bridge knots,  torus knots, and alternating knots,  in  $\calc$.  A few references include~\cite{MR2366190, MR4198504, MR3657225}.  

From a geometric perspective, understanding the four-genus of knots,  $g_4(K)$, was one of  the early motivations of developing the concordance group, and the induced function $g_4\co \calc   \to \Z$ continues to be studied.  One highlight of this has been the study of differences of torus knots, $g_4(T_{p,q} \cs -T_{p',q'})$; see, for instance,~\cite{MR4198504, feller2020note, MR2975163}.

\medskip
\noindent{\bf Examples.}

A few examples will clarify the issues underlying the project here.  We let $T_{2,2k+1}$ denote the $(2,2k+1)$--torus knot and lew  $\calt_{2,2k+1}$ denote its concordance class.

\begin{itemize}
\item   The classes $\calt_{2,5}$ and $ \calt_{2,11}$ are, in a sense,  close to linear dependent, since $d(2\calt_{2,5}, \calt_{2,11}) =1$.     On the other hand,   $\calt_{2,7}$ and $ \calt_{2,11}$   are further from being dependent:   $d(a \calt_{2,7}, b \calt_{2,11}) \ge 2$,  for all non-zero $a$ and $b$; these are simple consequences of the results of Sections~\ref{sec:geo2k} and~\ref{sec:tpq}.

\item  Issues related to the triangle inequality are illustrated by the following.  For all $a$ and $b$ nonzero, the minimum of  $d(a \calt_{2,41}, b \calt_{2,61}) $ is $2$, realized by $d( 3\calt_{2,41} ,2\calt_{2,61}) = 2$; the minimum of   $d(a \calt_{2,61}, b \calt_{2,91}) $ is $2$, realized by $d( 3\calt_{2,61} ,  2\calt_{2,91}) = 2$; and   the minimum of  $d(a \calt_{2,41}, b \calt_{2,91}) $ is $5$, realized by $d( 2\calt_{2,41},  \calt_{2,91}) = 5$.  This is discussed in Section~\ref{sec:failure} and elsewhere.

\item   The failure of transitivity that appears in defining $\P(G)$ is illustrated with the cyclic group $\Z_6$.  The pair of elements $\{1, 2\}$ have a nonzero common multiple, as does the pair $1$ and $3$.  Yet the pair $\{2,3\}$ does not have a common nonzero multiple.  In the context of $\calc$, we have  $2\calt_{2,3}$ is in an infinite number of distinct cyclic subgroups, generated by elements of the form $\calt_{2,3} \cs \calj$ for arbitrary 2--torsion elements $\calj$.  Thus, for instance, determining $\Delta(\calt_{p,q}, \calt_{p',q'})$ entails determining the minimum of $g_4( aT{p,q} \cs bT_{p',q'} \cs J)$ for all nonzero $a$ and $b$ and for all knots $J$ of finite order in $\calc$.

\end{itemize}
\medskip
\noindent{\bf Notation.}  It will be important to distinguish between a knot and the concordance class represented by the knot.  To do so, we will change typeface; for instance, for the knot $K \subset S^3$ we will  write $\calk \subset \calc$ for the concordance class represented by $K$.  Later, when we place an equivalence relation on $\calc$ to form  projective space, $\P(\calc)$, the equivalence class of $\calk \in \calc$ will be denoted $[\calk] \in \P(\calc)$.  

The functions $g_4$ can be defined on  the  set of knots or  on $\calc$.  Similarly, we can refer to the distance $d(K,J)$ or $d(\calk, \calj)$, working with knots or concordance classes.  This ambiguity should not be problematic, so we do not add to the notation to distinguish these functions based on their domains.

\medskip

\noindent{\bf Outline.} Sections~\ref{sec:geo2k}~and~\ref{sec:tpq} discuss a basic family of  examples, that of two-stranded torus knots, $T_{2,2k+1}$.  In this setting we explore the problem of minimizing $d(aT_{2k+1}, bT_{2n+1})$ over all $a\ne 0$ and $ b \ne 0$.  The first of these two sections concerns geometric tools that provide upper bounds; the second section uses signature functions to find lower bounds.  We will define $\overline{d}(K, J) = \min \{ d(aK, bJ )\ \big| \ a, b \ne 0\}$.  Elementary consequences of the work in these sections are the following results presented in Section~\ref{sec:growth}:

\begin{itemize}
\item $\overline{d} (\calt_{2,2k+1}, \calt_{2,2n+1})$ goes to infinity for fixed $k$  as $n$ grows. The  growth rate is of  order $ {n}/{2k}$.

\item For fixed  $N>0$,  the set of values  $\overline{d}(\calt_{2,2k+1}, \calt_{2,2n+1})$ for all $k$ and $n$ satisfying $\big| n - k \big| \le N$ is bounded.

\item The function  $\overline{d} \co \calc \times \calc \to \Z$ does not satisfy the triangle inequality.

\end{itemize}

In Section~\ref{sec:project} we  develop the formal algebra  that permits us to define the appropriate equivalence relation and form the quotient space $\P(\calc)$.  In our setting,  we could restrict our algebraic discussion to general abelian groups,  but the natural map $\calc \to \calc \otimes \Q$ leads us to consider   projectivizing modules over $\Q$, and ultimately modules $M$ over arbitrary integral domains.   In Section~\ref{sec:metrics} we define a canonical metric on $\P(M)$ for a given integer-valued metric on $M$ and Section~\ref{sec:2kc} considers properties of that metric.

In Section~\ref{sec:distance1} we discuss tools for computing a   pseudo-metric $\delta$ that is initially defined on $\P(\calc)$.  This discussion is based on the signature function, as developed in Section~\ref{sec:tpq}.  

Section~\ref{sec:small ball} concerns the metric $\Delta\co \P(\calc) \times \P(\calc) \to \Z$.  Computations are much more difficult, so we restrict ourselves to the setting of \t2k-torus knots and the analysis of small balls in the subset of $\P(\calc)$ consisting of elements represented  by these knots.  It is worth mentioning now that for any class $\calk $ in the span of such torus knots, there are concordance class $\calj \in \calc$ satisfying  $ \calj  \in [\calk]$ but for which $\calj$ is not represented by a knot in the span of torus knots.   Here is a simple example. Let $\calj $ be any element of order two in $\calc$.  Then $2(\calt_{2,3} \cs \calj) = 2\calt_{2,3}$ so both represent the same element in $\P(\calc)$;  however, $\calt_{2,3} \cs J$ is not in the span of torus knots.  

 A challenging problems asks,  for each subgroup   $\cals \subset \calc$,  whether the inclusion $\P(\cals) \to \P(\calc)$ is an isometry. This and other problems are summarize in Section~\ref{sec:problems}.
 
 \medskip
 
 \noindent{\it Acknowledgements.}  I thank Samantha Allen for her valuable feedback on early drafts of the manuscript. 

%%%%%%%Section%%%%%%%%%%%%%%

\section{Torus knots, $T_{2,2k+1}$:   geometric constructions.}\label{sec:geo2k}

The space $\P(\calc)$ provides a structure for exploring knot concordance, but ultimately the questions that arise can only be understood by constructing surfaces in $B^4$ bounded by connected sums of the form  $aK \cs -bJ$ and then  finding lower bounds on the genus of such surfaces. In this section we consider the construction of  surfaces bounded by linear combinations of pairs $\{T_{2,2k+1}, T_{2,2n+1}\}$ and then focus on the  special case of minimizing $g_4(T_{2,2k+1} \cs - \beta T_{2,2n+1})$.  In the next section we use the signature function to show that minimizing $g_4(a T_{2,2k+1} \cs - b  T_{2,2n+1})$ over all $a$ and $b$ can be reduced to a finite set of such pairs, and completely  resolve the problem of minimizing $g_4( T_{2,2k+1} \cs - b  T_{2,2n+1})$.  We also provide examples for which the overall minimum over all $a$ and $b$  is not achieved with $a =1$.

At this point we note an  interesting aspect of the family of $(2,2k+1)$--torus knots:  from the perspective of Heegaard Floer theory, they   all appear to be linearly  dependent: there is a chain homotopy equivalence  $\cfk^\infty(nT_{2,2k+1}) \oplus A_1 \simeq\cfk^\infty(kT_{2,2n+1}) \oplus A_2$ for some pair of acyclic complexes $A_1$ and $A_2$.  More generally, Feller and Krcatovich~\cite{MR3694648} have demonstrated the limited ability of Heegaard Floer Upsilon to obstruct linear dependance in $\calc$ among general torus knots.  If one moves to the realm of involutive Heegaard Floer theory, some limited results concerning torus knots $T_{2,2k+1}$ become available; see~\cite{MR3649355}.

\subsection{Basic Construction}

A Seifert surface for  $T_{2,2k+1}$ is constructed  by attaching $2k$ once twisted bands to a disk.  Figures~\ref{fig:schematic1} and~\ref{fig:schematic2} are schematics representations of Seifert surfaces for $T_{2, 13} \cs -2T_{2,5}$ and $T_{2, 13} \cs -3T_{2,5}$.  Curves drawn on the surface  represent unlinks with 0 framing on the surfaces.  (In the first  illustration there are eight such curves; each one goes over a band on the top and a band on the bottom.  In the second illustration there are 10 such curves.)  The surfaces can be surgered in $B^4$ to yield surfaces of lower genus than the Seifert surfaces, thus giving upper bounds on the four-genus of the knots.  

\begin{figure}[h]
\labellist
%\pinlabel {\text{\Large{$J_1$}}} at 47 30
%\pinlabel {\text{\Large{$J_2$}}} at 47 110
%\pinlabel {\text{\Large{$-J_2$}}} at 204 30
%\pinlabel {\text{\Large{$-J_1$}}} at 205 110
\endlabellist
\includegraphics[scale=.9]{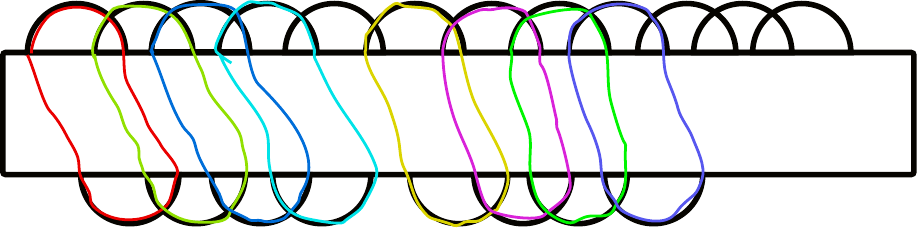} 
\caption{A schematic of a Seifert surface for $T_{2,13} - 2T_{2,5}$ and surgery curves.}
\label{fig:schematic1}
\end{figure}

\begin{figure}[h]
\labellist
%\pinlabel {\text{\Large{$J_1$}}} at 47 30
%\pinlabel {\text{\Large{$J_2$}}} at 47 110
%\pinlabel {\text{\Large{$-J_2$}}} at 204 30
%\pinlabel {\text{\Large{$-J_1$}}} at 205 110
\endlabellist
\includegraphics[scale=.9]{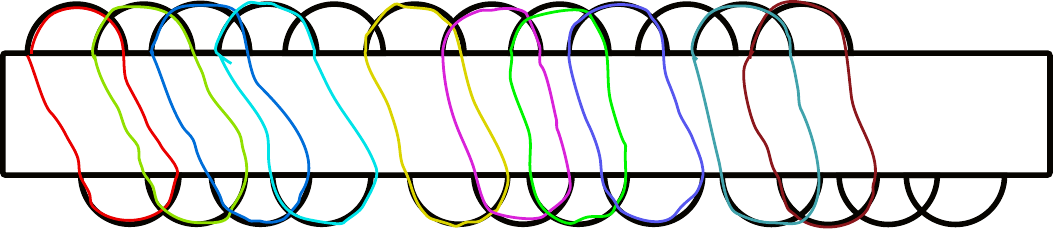} 
\caption{A schematic of a Seifert surface for $T_{2,13} - 3T_{2,5}$ and surgery curves.}
\label{fig:schematic2}
\end{figure}

\subsection{Application to $T_{2,2n+1} -aT_{2,2k+1}$.}

The two figures illustrate two possibilities   for connected sums $T_{2,2n+1} \cs - aT_{2,2k+1}$, where $n>k$.  In each case, there is an elementary computation of the genus of the resulting surface in $B^4$.  With a bit of experimenting, one might suspect that the minimum four-genus of $T_{2,2n+1} \cs - aT_{2k+1}$ is achieved when $a$ is close to  $\alpha = \lfloor\frac{2n+1}{2k+1}\rfloor$.  Here we present the computation for the two values of $a$ closest to $\alpha$. Once we consider signatures, we will prove that one of these surfaces    realizes  the minimum.  As mentioned earlier, it is not always the case the $\min \{ d(aT_{2,2k+1}, bT_{2n+1})\ \big| \ a, b \ne 0\}$ is realized  with $b= 1$.

\begin{theorem} \label{thm:b1}  For $0 < k <n$,  let $\alpha = \lfloor\frac{2n+1}{2k+1}\rfloor$.\begin{enumerate}
\item  $ g_4(T_{2,2n+1} \cs - \alpha T_{2k+1} ) \le n - \alpha k.$
\item If $2k+1$ does not divide $2n+1$, then  $ g_4(T_{2,2n+1} \cs - (\alpha +1) T_{2k+1} )  \le   \alpha  (k+1)  +k- n .$
\end{enumerate}
\end{theorem}

\begin{proof}
The genus of the surface before the surgery is one half the total number of bands.  Surgery reduces the genus by the number of surgery curves, which is one half the number of bands that have surgery curves going over them.  We call the bands that don't have surgery curves going over them  {\it free bands}.  Thus, the genus after surgery is one half the number of free bands.  For instance, in Figure~\ref{fig:schematic1} there are four such bands, all on top.  In Figure~\ref{fig:schematic2} there are also four such bands, two on top and two on the bottom.

In the first case, all the free bands are on top.  There are two types: those resulting from gaps and those at the end.  The count is $(\alpha -1) + (2n - \big(\alpha (2 k+1) -1)\big)$
 
In the second case, there are $\alpha  $ free bands   on the top (this uses the fact that $2k+1$ does not divide $2n +1$).  On the bottom there are $2k - \big( 2n  -  \alpha (2k+1) \big) $ free bands.  The result now follows from an algebraic simplification.
\end{proof}

\section{Torus knots, $T_{2,2k+1}$:   signature results.}\label{sec:tpq}

The signature function provides strong bounds on the four-genus of knots.  For general torus knots these bounds are generally weaker than those attained through gauge theoretic methods; for instance,  in~\cite{MR1241873} Seiberg-Witten theory  was used to  resolve the  Milnor Conjecture related to $g_4(T_{p,q})$.  Heegaard Floer theory similarly applies in this setting~\cite{MR2026543}, as do methods arising from Khovanov invariants~\cite{MR2729272}.  On the other hand, many results  concerning difference of torus knots seem to be unavailable to methods other than the signature; see~\cite{MR3784267} for a further discussion.

\subsection{Signature functions}
In order to simplify our calculations, rather than work with the (two-sided averaged) Levine-Tristram signature function, $\sigma_K(\omega)$, we will normalize and define $\sigma'_K(\omega) =-\sigma_K(\omega) /2 $.  To further simplify notation, instead of working with  unit complex numbers on the upper half-circle, $\omega$, we will change variables so that domain is  $[0,1]$ by setting  $\omega = e^{ \pi i t }$.  The next well-known result follows from the work of Tristram~\cite{MR0248854} and Viro~\cite{MR0370605}.

\begin{proposition}    For all knots $K$ and for all $t \in [0,1]$, $g_4(K) \ge \big| \sigma'_K(t)\big|.$
\end{proposition}

In light of this, we define the function that maximizes this bound.

\begin{definition}
For a knot $K$, $S(K) = \max_{0 \le t \le 1}\{\big| \sigma'_K(t) \big| \}$.  
\end{definition}

\subsection{Signature functions of $T_{2,2k+1}$.}  A standard result for $2$--stranded torus knot is as following. 
\begin{proposition}  
If $1 \le j \le k$ and $ \frac{2j-1}{2k+1} < t < \frac{2j+1}{2k+1}$, then $\sigma'_{T_{2,2k+1}}(t) = j$. If $t < \frac{1}{2k +1}$  then $\sigma'_{T_{2,2k+1}}(t) = 0$
\end{proposition}

It is clear from Theorem~\ref{thm:b1} that in studying the difference $   bT_{2n+1} \cs -aT_{2,2k+1}$, special care is required in the special case that $\frac{2n+1}{2k+1}$ is an integer.  Because of this, the  floor function arises naturally in the calculations,  but it has to be slightly modified.   

\begin{definition} We will write   $\lfloor \lfloor x \rfloor \rfloor$  for the function that equals $  \lfloor x \rfloor $  if $x$ is not an integer and  $  \lfloor x \rfloor - 1$  if $x \in \Z$.  Alternatively, $\lfloor \lfloor x \rfloor \rfloor = -   \lfloor -x \rfloor  -1$. 
\end{definition}

A simple calculation now  yields the following result.

\begin{corollary}\label{corrr:sig} Let $K =bT_{2,2n+1} \cs - aT_{2,2k+1}$ with $0< k<n$ and $a, b >0$.    For sufficiently small $\epsilon$ with $\epsilon >0$ we have:
\begin{enumerate}
\item  $\sigma'_K(\frac{1}{2k+1} - \epsilon) =  a \lfloor \lfloor \frac{1}{2}\frac{2n+1}{2k+1} +\frac{1}{2} \rfloor\rfloor \ge a$.\vskip.1in

\item  $\sigma'_K(1) = bn  -ak  $.\vskip.1in

\item  $\sigma'_K(\frac{2k-1}{2k+1} +  \epsilon) =bn -ak    -b \lfloor  \lfloor \frac{2n+1}{2k+1}  \rfloor \rfloor$.

\end{enumerate}
\end{corollary} 

\begin{corollary}
For fixed values of $k$ and $n$ with $0 < k <n$ and for every $N >0$, the set of pairs $a $ and $b$ such that $S(bT_{2, 2n+1} \cs -  aT_{2,2k+1}) \le N $ is finite.
\end{corollary}
\begin{proof}Suppose that  $S(bT_{2, 2n+1} \cs -  aT_{2,2k+1}) <N $. Condition (1) of Corollary~\ref{corrr:sig} implies that $0< b \le  N$. For each value of $b$ in that interval, Condition (3) implies that the set of possible values of $a$ is also finite.
\end{proof}

\subsection{The case of $b=1$:  minimizing $g_4(T_{2,2n+1}  \cs -  aT_{2k+1})$}

\begin{theorem}\label{thm:a=1}  Let $\alpha =  \lfloor  \frac{2n+1}{2k+1}   \rfloor$.  The minimum value of $g_4(T_{2,2n+1} - aT_{2k+1})$ is achieved when $a = \alpha$ or $\alpha +1$, with the two possible values given by 
\begin{itemize}
\item  $g_4(T_{2,2n+1} \cs -  \alpha T_{2k+1})=  n - \alpha k $.\vskip.1in

\item  $g_4(T_{2,2n+1} \cs -  (\alpha+1)  T_{2k+1}  )=  (\alpha +1)(k+1) - n -1 $.\vskip.1in
\end{itemize}
\end{theorem}

\begin{proof}
Let $F(a)$ denote the signature function bound on  $g_4(T_{2,2n+1}  \cs -  a T_{2k+1})$.  We consider two cases:  $a \le \alpha$ and $a \ge \alpha +1$.\medskip

\noindent{\bf Case 1: $a \le \alpha$.}  In this case, we consider the signature at $t = 1$.   Noting that $n-ak \ge 0$ we find that $F(a) \ge n -ak$.  This bound is realized at $a = \alpha$, so we have $F(\alpha) = n - \alpha k$ and $F(a) > F(\alpha)$ for all $a < \alpha$.\medskip

\noindent{\bf Case 2: $a \ge \alpha +1$.}  In this case, we consider the signature at $t = \frac{2k-1}{2k+1} + \epsilon$ for some small $\epsilon$.  In this case we have 
$F(a) \ge  ak  - n +   \lfloor  \frac{2n+1}{2k+1}   \rfloor$.  This bound is realized when $a = \alpha +1$.

Now we can combine these two cases.  If $S(\alpha) \le S(\alpha +1)$, then $S(\alpha) \le S(a)$ for all $a$ and  $g_4(T_{2,2n+1} \cs -  a T_{2k+1})$ is minimized at $a - \alpha$. Similarly if $S(\alpha + 1) \le S(\alpha)$.  
\end{proof}

\subsection{Basic examples} We begin with a few examples.

\begin{example} $\mathbf{ \text{\bf Min} \{ g_4(bT_{2,23} \cs -  aT_{2,7})\} = g_4( T_{2,23}  \cs -  3T_{2,7}) = 2.}
$ 
In this case we have  $n=11$ and $k =3$.  Then applying Corollary~\ref{corrr:sig} we find 
\[ 
g_4( bT_{2,23} - aT_{2,7}) \ge S(bT_{2,23} \cs -  aT_{2,7}) \ge  \max\{  \big| 2b \big|,   \big| 11b - 3a\big|, \big| 9b - 3a \big|\}.
\]

By Theorem~\ref{thm:b1} we have $g_4(T_{2,23} \cs -  3T_{2,7} )\le 2$.  Thus, the minimum for $g_4(bT_{2,23} - aT_{2,7} )$ must be realized by $b=1$.  In this case, we have 
\[ 
g_4(  T_{2,23} - aT_{2,7}) \ge S( T_{2,23} \cs -  aT_{2,7})\} \ge  \max\{  \big| 2 \big|,   \big| 11  - 3a\big|, \big| 9 - 3a \big|\}.
\]
The only value of $a$ for which this maximum is 2 is $a=3$.
\end{example}

\begin{example} $\mathbf{ \text{\bf Min}  \{g_4(bT_{2,17} \cs -  aT_{2,11}) = g_4( 2T_{2,17} \cs -  3T_{2,7}) = 2.}
$ Here we present a case in which $g_4(bT_{2,2n+1} \cs -  aT_{2k+1})$ is not realized by  $g_4(T_{2,2n+1} \cs -  a T_{2k+1})$ for any $a$.  
Let  $k=5$ and  $n=8$.     By Corollary~\ref{corrr:sig}  if we consider combinations of the form $ T_{2,17} \cs -  aT_{2,11}$ we have 
\[ S(T_{2,17} \cs -  aT_{2,11}) \ge \max \{ 1, \big|8 - 5a\big|, \big|7 - 5a\big| \}.
\]
For all values of $a$, this is always at least 3.  

For general $b$ we have
\[ S(bT_{2,17} \cs -  aT_{2,11}) \ge \max \{ b, \big|8b - 5a\big|, \big|7b - 5a\big| \}.
\]  Clearly, if $b\ge 3$ then the maximum is at least 3.

We next observe that  $g_4(2 T_{2,17} \cs -  3T_{2,11}) = 2$.  If we draw a schematic for this difference, there are two groups of 16 bands on the top and 3 groups of 10 bands on the bottom.  All 10 bands of the bottom left group can be surgered, as can the 10 bands of the bottom right group.  This leaves 5 bands free on each of the top two groups.  These can be combined with bands on the bottom central group to perform surgery on 9 more curves.  Thus, we have reduced the genus by $10 + 10 +9 = 29$.  Finally, $31-29 = 2$, yielding the minimum. 

If we consider combinations of the general  form $b T_{2,17} \cs -  aT_{2,11}$, then the signature   at $\frac{1}{11} - \epsilon$ is $b$, so we see that we need consider only the final case of $2 T_{2,17} \cs -  aT_{2,11}$.  In this case the signature   at $\frac{1}{11} - \epsilon$ is 2, so we will not find a value of $a$ for which the four-genus is 1. 

Figure~\ref{fig:schematic3}  illustrates two   of the signature functions that arise. (The $t$ axis is labelled from 0 to 1,000, indicating that the signature function was evaluated at points $i/1000$.)

\begin{figure}[h]
\labellist
%\pinlabel {\text{\Large{$J_1$}}} at 47 30
%\pinlabel {\text{\Large{$J_2$}}} at 47 110
%\pinlabel {\text{\Large{$-J_2$}}} at 204 30
%\pinlabel {\text{\Large{$-J_1$}}} at 205 110
\endlabellist
\includegraphics[scale=.2]{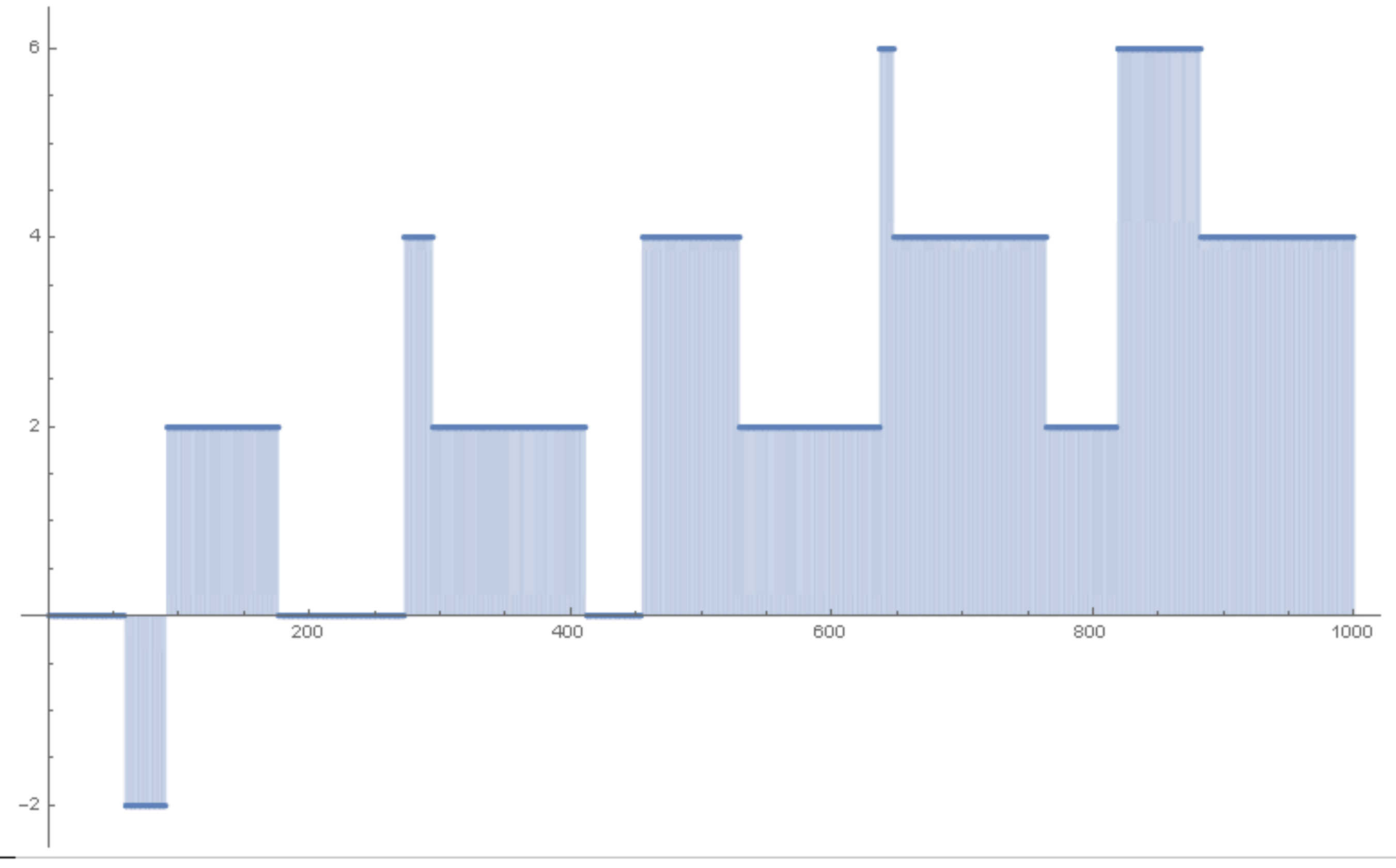} \hskip.5in   \includegraphics[scale=.2]{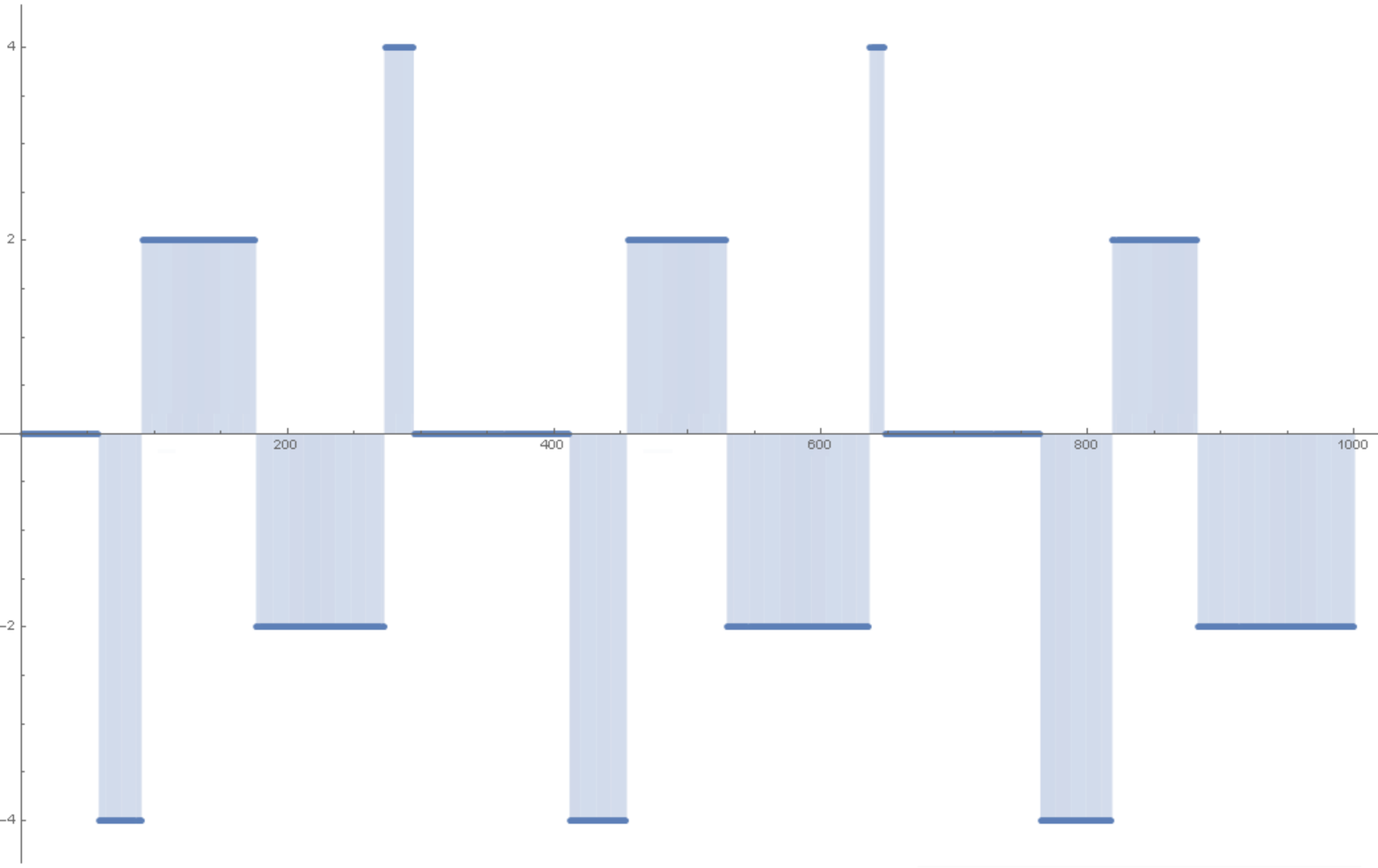} 
\caption{$\sigma_K(\omega)$ for $K =T_{2,17} \cs - 2T_{2,11}$ and  $K =2T_{2,17} \cs - 3T_{2,11}$.}
\label{fig:schematic3}
\end{figure}
\end{example}

\subsection{Failure of the triangle inequality}\label{ex:91}  Our final example is the most technical.  It will be used later to demonstrate the failure of the triangle inequality.

\begin{enumerate}

\item  Min$\{ g_4( bT_{2,61} \cs -  aT_{2,41})\} = 2$, realized by $g_4(2 T_{2,61}\cs -  3T_{2,41}) =2$

\item   Min$\{ g_4(bT_{2,91}\cs -  aT_{2,61})\} = 2$, realized by $g_4(2T_{2,91} \cs -  3T_{2,61}) = 2$.

\item  Min$\{ g_4(bT_{2,91}\cs -  aT_{2,41})\} = 5$, realized by $g_4( T_{2,91} \cs -  2T_{2,41}) = 5$.

\end{enumerate}

We now work through each case. 
\begin{enumerate}

\item  A construction similar to the one used to show  $g_4(2 T_{2,17} \cs -  3T_{2,11}) =2$ demonstrates that   $g_4(2 T_{2,61} \cs -  3T_{2,41}) \le 2$.  Thus, we need to show that 2 is the minimum.  Here $n= 30$ and $k=20$.  Applying Corollary~\ref{corrr:sig} we find 
\[
S(bT_{2,61} \cs -  aT_{2,41}) \ge  \max\{  \big| b \big|,   \big| 30b - 20a \big|, \big| 29b - 20a \big|\}.  
\]
It is now a trivial exercise to show this has minimum 2, realized when $b =2$ and $a=3$.
\item   
Showing that   $g_4(2T_{2,91} \cs -  3T_{2,61}) =2$ demonstrates that   $g_4(bT_{2,91}\cs -  aT_{2,61}) \le 2$.  Thus, we need only show that 2 is the minimum.

Here $n= 45$ and $k=30$.  Applying Corollary~\ref{corrr:sig} we find 
\[
S(bT_{2,91} \cs -  aT_{2,61}) \ge  \max\{  \big| a \big|,   \big| 45b - 30 a \big|, \big| 44b - 30a \big|\}.  
\]
It is again a trivial exercise to show this has minimum 2, realized when $b =2$ and $a=3$.

\item  The basic construction shows that   $g_4(  T_{2,91} \cs -   2T_{2,41}) \le 5$.   Thus, we need to show that 5 is the minimum.

Here $n= 45$ and $k=20$.  Applying Corollary~\ref{corrr:sig} we find 
\[
S(bT_{2,91} \cs - aT_{2,41}) \ge  \max\{  \big| 2b \big|,   \big| 45b - 20 a \big|, \big| 44b - 20a \big|\}.  
\]
We leave it to the reader to check that the minimum is 5,  realized when $b =1$ and $a=2$.

\end{enumerate}

 \subsection{The growth of $\overline{d}(\calt_{2,2k+1}, \calt_{2n+1})$}\label{sec:growth}    Let $\overline{d}(\calk, \calj) = \min \{ d(a\calk , b\calj) \ \big| \ a \ne 0 \ne b\}$.
 In Section~\ref{sec:2kc} we will describe in detail the distance $\delta$ on the projective space $\P(\calc)$ and will see that in the following theorem statement, $\overline{d}$ can be replaced with $\delta$.
 \begin{theorem}\label{thm:growth2}  For any fixed integer $k \ge 1$, 
 \[ 
 \lim_{n \to \infty}  \frac{\overline{d}(\calt_{2,2k+1},\calt_{2n+1})  }{n} = \frac{1}{2k+1}.
 \]
 \end{theorem}
\begin{proof} By Theorem~\ref{thm:b1} we have that 
\[
\overline{d}(  \calt_{2,2k+1},\calt_{2n+1})  \le  \min\{   n - \alpha k ,  (\alpha +1)(k+1) - n -1 \},
\]
where $\alpha = \lfloor \frac{2n+1}{2k+1} \rfloor$.  Since $\big| \alpha - \frac{2n+1}{2k+1} \big| \le 1$ and is  not multiplied by $n$ in this bound, we can replace $\alpha$ with $\frac{2n+1}{2k+1}$ in the bound without changing the limiting behavior once it is divided by $n$.  An elementary algebraic manipulation then provides the upper bound for the  limit of $\overline{d}/n$ to be  $\frac{1}{2k+1}$.

To get the lower bound, we use Corollary~\ref{corrr:sig}, which implies 
\[\overline{d}(   \calt_{2,2k+1},\calt_{2n+1})) \ge  \lfloor \lfloor \frac{1}{2}\frac{2n+1}{2k+1} +\frac{1}{2} \rfloor\rfloor.
\]
Again, the floor function differs from its argument by an amount that bounded by 1, so we have 
\[\overline{d}(   \calt_{2,2k+1},\calt_{2n+1})) \ge    \frac{1}{2}\frac{2n+1}{2k+1} +\frac{1}{2}  .
\]
Forming the quotient with $n$ and taking the limit as $n$ goes to infinity gives the desired lower bound.
\end{proof}

%%%%%%%Section%%%%%%%%%%%%%%

\section{Projectivizing abelian groups}\label{sec:project}

We would like to define a distance on the concordance group by something like 
\[ \min \{ d(K, J) \ \big| \ \calk \in \cals_1, \calj \in \cals_2 \},
\]
where $\cals_1$ and $\cals_2$ are maximal cyclic subgroups of $\calc$ containing $\calk$ and $\calj$, respectively.  Such maximal subgroups exist by Zorn's Lemma    however  they need not be unique.  For example, consider $\Z \oplus \Z_2$.  The subgroups $\left< (1,0) \right>$ and  $\left< (1,1) \right>$ are both maximal cyclic subgroups containing $(2,0)$. 

In this section we will discuss a general approach to the algebra associated to the relation on a group generated by the property of elements being in a common cyclic subgroup.  The construction is modeled on that of projective spaces associated to vector spaces. Although our interest is ultimately in the abelian group $\calc$, a $\Z$--module, it will be valuable to work with general modules over integral domains.

\subsection{The projective space of an $R$--module}

Let $R$ be an integral domain,   let $M$ be a left $R$--module, and let $M^\circ = M \setminus 0$.  Define a binary relation on $M^o$ by $x \sim' y$ if there exists an $m\in M$ such that $x = rm$ and $y= sm$ for some $r, s \in R$ and some $m\in M$.
Notice that this is reflexive and symmetric, but it need not be transitive.
\begin{example}
Let $R = \Z$ and $M= \Z \oplus \Z_2 \oplus \Z_2.$  Then $(1,1,0) \sim' (2,0,0)$ and   $(1,0,1) \sim' (2,0,0)$, but  $(1,1,0) \not\sim' (1,0,1).$   
\end{example}

\begin{definition}  Define the {\it projective relation} on $M^\circ $, $x \sim_M y $, to be the equivalence relation generated by $\sim'$.  That is, $a \sim_M b$ if and only if there is a finite chain
\[ x = x_0 \sim' x_1 \sim' \cdots \sim' x_n =  y.
\]
Except where needed, we will write ``$\sim$'' instead of ``$\sim_M$''.
\end{definition}

The following result might clarify the equivalence relation and highlights why we chose to define it in terms of having common divisors instead of having common multiples.

\begin{theorem} \label{thm:infinite}
If $x, y \in M$ are non-torsion elements and $[x] = [y] \in \P(M)$, then there exist elements $a, b \in R$ such that $ax = by \ne 0$.\end{theorem}

\begin{example}\label{ex:6} The classes $[2] = [3] \in \P(\Z_6)$, but the  two elements $2, 3 \in \Z_6$ have no nonzero multiple in common.
\end{example}

\begin{definition} Define the projective space $\P( M) = M^\circ /\sim$.  Set $\P^*(M) = * \sqcup \P(M)$, where $*$ denotes a disjoint point.
\end{definition}
We have the following     elementary result.   

\begin{theorem} \label{thm:inject} Let  $M$ be an $R$--module and $N$ be an $S$--module, let  $\phi \co R\to S$ be  a  ring homomorphism,  and let $F\co M \to N$ be a   module homomorphism over $\phi$.  Then  (1)    $\F$  induces a map  $F_*\co  \P^*(M)  \to \P^*(N)$, sending the equivalence class of $x$ to $*$ if $rx \in \ker (F)$ for some $r \ne 0$; (2)  If $F$ is surjective, then $F_*$ is surjective;  (3)  If $F$ is injective, then   there is also an induced map   $F_* \co \P(M) \to    \P(N)$; this map need not be injective.
\end{theorem} 

\begin{example} The projective space  $\P(\Z_2 \oplus \Z_2)$ has three elements corresponding to the three nontrivial elements in the group.  On the other hand, $\P(\Z_2 \oplus \Z_2 \oplus \Z_3)$ has one element; this follows from Theorem~\ref{thm:tor1}  below or an elementary calculation.   This example  shows that in Theorem~\ref{thm:inject} we cannot conclude that an injective map on $M$ induces an injective map on $\P(M)$.
\end{example}

\begin{example}
If $R = \F$ is a field,  then $\P (M)$ is the standard projective space.  For instance, if $M = \F^n$, then $\P(\F^n) $ is, in the usual notation, $\P \F^{n-1}$.  In Corollary~\ref{corollary:free} we discuss the case that $M = R^n$ for an integral domain $R$. 
\end{example}

\subsection{Torsion groups}

\begin{example}  For any $n>1$, $\P(\Z_n)$ has one point, the equivalence class of $1$.
\end{example}

\begin{theorem}\label{thm:tor1}  If $G$ is a torsion abelian group containing elements $a$ and $b$ of distinct prime orders $p$ and $q$, then $\P(G)$ has one element.
\end{theorem}
\begin{proof} Given an arbitrary $x \ne 0 \in G$, by taking a multiple we see that  $x$ is equivalent to some element $x'$  of prime order $s$.  Assume $s \ne q$.  Then 
\[ 
x' \sim' qx' = q(x'+b) \sim' s(x' +b) = sb \sim' b.
\]
Thus, every element is equivalent to either $a$ or $b$.  But these are also equivalent: \[ 
a \sim' qa = q(a +b) \sim' p(a +b) = pb \sim' b.
\]
\end{proof}

\begin{theorem}\label{thm:tor2}  Suppose that $p$ is a prime and each elements of an abelian group $G$ has order $p^k$ for some $k$.  Let $H \subset G$ be the subgroup consisting of elements $x$ satisfying $px = 0$.

Then the map induced by  inclusion    $\P(H) \to \P(G)$ is a bijection.
\end{theorem}
\begin{proof}  Notice that $H$ is a $\Z_p$--vector space and thus each non-zero element   lies on a unique line, or stated  equivalently, in a unique cyclic subgroup. 

Given an element $x  \ne 0 \in G$, choose the least $n > 0$ such that $nx \in H$ and note $nx \ne 0.$ Denote  $nx $ by $F(x)$.   We claim that $F$ induces a bijection  $F_*\co \P(G) \to \P(H)$.  

First we show it is well-defined.  If $x \sim' y$ then $x$ and $y$ are in a common cyclic subgroup of $G$.  The intersection of that subgroup with $H$ is a cyclic subgroup, and so $F(x)$ and $F(y)$ lie on a common line in  $H$, which must be the unique line through $F(x)$.  Continuing in this way, if there is a sequence $x = x_0 \sim' x_1\sim'  \cdots \sim' x_n = y$, we see that $F(x_i)$ lies on the line through $F(x)$ for all $i$ and in particular $F(x)$ and $F(y)$ lie in a common cyclic subgroup.  Thus, $F^*$ is well-defined.

It is clear that $F_*$ is surjective.  

For injectivity, first note that it is  evident  that if $F(x) \sim_H F(y)$ then $F(x) \sim_G F(y)$.  It is also clear that $F(x) \sim_G x$ and $F(y) \sim_G y$.  So, if $F(x) \sim_H F(y)$ we have the chain
\[ x \sim'_G F(x) \sim'_G F(y) \sim' y.
\]
\end{proof}
\begin{example}
For   direct sums, infinite as well as finite, and for any prime $p$,  the inclusion $\oplus_i \Z_{p} \to \oplus_i \Z_{p^{a_i}}$ induces a bijection $\P(\oplus_i \Z_{p}) \to \P(\oplus_i \Z_{p^{a_i}})$.  The domain is     a $\Z_p$--projective space.  There is one point in  $\P(\oplus_i \Z_p)$ for each order $p$ cyclic subgroup.  

In the case of $p=2$, cyclic subgroups correspond to nontrivial elements of  $\P(\oplus_i \Z_2)$ and thus there is a bijection $( \oplus_i \Z_2 \setminus 0 ) \to \P(\oplus_i \Z_{2^a_i})$. 

In the case of a finite sum, $  \oplus_{i=0}^n \Z_{p^{a_i}}$ we have that the number of elements in the projective space is $(p^n -1)/(p-1)$.
\end{example}

\subsection{The torsion free case}

Let $M$ be a torsion free module over $R$ and let $\Q(R)$ denote the field of fractions.   Let $M_\Q = M \otimes \Q(R)$ be the associated $\Q(R)$ vector space.   

\begin{theorem}\label{thm:free2} If $M$ is torsion free, then there is  a natural bijection $\P(M) \to \P (M_\Q)$.
\end{theorem}

\begin{proof}  We first recall the elementary fact  that  $M$ is torsion free implies that $M \to M_\Q$ is injective.  Another elementary observation is that for every  $x  \ne 0 \in M_\Q$, there is an $ \alpha \ne 0  \in R$ such that $\alpha x \in M$.   By Theorem~\ref{thm:inject} there is a natural   map $\psi \co \P(M) \to \P (M_\Q)$.

It is clear that  $\psi$ is surjective:  $m \otimes \frac{a}{b} \sim' b(  m \otimes \frac{a}{b} ) = m \otimes a =  am \otimes 1$. 

To show that $\psi$ is injective, suppose that $a, b \in M$ and $a \sim_{M_\Q} b$.  Then there are an  $r, s \in \Q(R)$ and an $m \in M_\Q$ such that $a= rm$ and $b=sm$.  Choose an element in $t \in R$ such that $tr \in R$,  $ts\in R$ and  $tm \in M$.    Then we have the following relations in $M$, where each element within parentheses is in $R$ or $M$. 
\[a \sim'_M (t^2)a =(t^2)(rm) = (tr)(tm)  \sim'_M (tm) \sim'_M (ts)(tm) =  (t^2)(sm) = (t^2)b \sim' b.
\] \end{proof}

\begin{corollary}\label{corollary:free}  The inclusion $\Z \to \Q$ induces a bijection $\P(\Z^\infty) \to  \P(\Q^\infty)= \Q\P^\infty$.
\end{corollary}

\subsection{Modules with free parts and torsion}
We continue to assume that $R$ is an integral domain.

\begin{theorem}\label{thm:tor2}  For arbitrary nonzero elements $a$ and $b$ in an   $R$--module $M$, if $a \sim b$ and $a$ is $R$--torsion, then $b$ is also $R$--torsion.
\end{theorem}
\begin{proof}
If $a \sim' b$ then there are   an $m$, $r$, and $s$ so that  $a = rm$ and $b = sm$.  There is an $\alpha \in R$ such that $\alpha \ne 0$ and  $\alpha a = 0$.  Thus $\alpha r m = 0$.  We then have that $\alpha r b = \alpha r sm =0$.  Since $\alpha \ne 0 $,  $r \ne 0$,  and $R$ is an integral domain, it follows that  $\alpha r \ne 0$.  Thus $b$ is also torsion.  

Finally, we see that  in any sequence \[a = x_0 \sim' x_1 \sim' \cdots \sim' x_n =  b, 
\]
each successive $x_i$ is torsion.
\end{proof} 
Let $\text{Tor}(M)$ denote the $R$--torsion submodule in the $R$--module $M$.  

\begin{theorem}\label{thm:tor3}
If $x \in M$ is not $R$--torsion and  $y \in M$ is   $R$--torsion, then $[x] = [x+y] \in \P (M)$.
\end{theorem}
\begin{proof}
Suppose that $r \ne 0$ and $r y = 0$.  Then $0 \ne rx = r(x +y)$ and
\[ x \sim'  rx  =  r(x +y)  \sim'   x+y .
\]  
\end{proof}

\begin{theorem} For any $R$--module $M$, $\P (M) = \P(\mathrm{Tor}(M)) \bigsqcup  \P(M /\mathrm{Tor}(M))$, where $\bigsqcup$ denotes disjoint union.

\end{theorem}
\begin{proof}

Let $\T(M) $ denote the set of classes in $\P(M)$ that are represented by elements in $ \mathrm{Tor}(M)$ and let $\F(M)$ consist classes in $\P(M)$ that are represented by  non-torsion elements of $M$.  If follows from  Theorem~\ref{thm:tor2} that $\P(M) = \T(M) \bigsqcup \F(M)$.  Thus we want to show that  $\P(\mathrm{Tor}(M)) = \T(M)$ and $\P(M / \mathrm{Tor}(M)) = \F(M)$.

\noindent{\bf Step 1.} Consider $a, b  \in \mathrm{Tor}(M)$. We first want to show that if $a \sim_M b$ then $a \sim_{\mathrm{Tor}(M)} b$.     Suppose that 
\[ a = x_0 \sim'_M x_1 \sim'_M \cdots \sim'_M x_n =  b 
\] 
is a chain.  We first note that  each $x_i \in {\mathrm{Tor}(M)} $.   If $a = rm $ and $x_1 = sm$, then since $a$ is torsion, $m$ is also torsion, and thus $x_1$ is torsion.  Proceed by induction.

We now need to show that if $a, b \in {\mathrm{Tor}(M)} $ and $a \sim'_M b$,  then $a \sim'_{\mathrm{Tor}(M)} b$.  Again, if $a = rm $ and $b = sm$, then since $a$ is torsion, $m$ is also torsion, and thus $a$ and $b$ are multiples of a common element in ${\mathrm{Tor}(M)}$. 

\noindent{\bf Step 2.}    We now observe that the previous step implies that   $\P( \mathrm{Tor}(M)) = \T(M) $.  The inclusion $\mathrm{Tor}(M)   \to M$ induces a map 
$\P( \mathrm{Tor}(M)) \to \T(M)$.  It is clearly surjective and the previous step show that it is injective.

\noindent{\bf Step 3.}    We now want to understand $\F(M)$.  Define $\phi \co \F(M) \to   \P(M /\mathrm{Tor}(M))$ by $[a]\to [\overline{a}]$, where $\overline{a}$ is the image of $a$ in $M / \mathrm{Tor}(M)$.  It is clear that   $\phi$ is well-defined: if $[a] = [a']$, then  $\overline{a} \ne 0 \ne  \overline{a}'$ and  $ [\overline{a}] = [\overline{a}']$. It is also clear that $\phi$ is surjective.  

We now prove injectivity.  If $\overline{a} \sim' \overline{b}$ then there is an element $m \in M$, elements $r, s \in R$, and torsion elements $t_1, t_2 \in \mathrm{Tor}(M)$ such that $a +t_1 = rm$ and $b +t_2 = sm$.  Suppose that $\alpha \in R$ satisifies $\alpha t_1 = 0 = \alpha t_2$.  Then $\alpha a = \alpha r m$ and $\alpha b = \alpha s m$.  We then have the chain
\[
a \sim' \alpha a \sim'm \sim'\alpha b \sim' b.
\] \end{proof}

\subsection{The projectivization of the concordance group: $\P(\calc)$}  We have the decomposition $\P(\calc) = \P({\text{Tor}}(\calc))  \sqcup \P(\calc/{\text{Tor}}(\calc))$.   

By Theorem~\ref{thm:free2} we have that  $\P(\calc/{\text{Tor}}(\calc))$ is in bijective correspondence with $\P(V)$ for some $\Q$--vector space.  Also, $\calc$ contains an infinite linearly independent set, so if fact,  $\P(\calc/{\text{Tor}}(\calc))$ is in bijective correspondence with $\P(\Q^\infty)$.

Since $\calc$ contains 2--torsion, if it also contains torsion of odd order, then by Theorem~\ref{thm:tor2}, $\P(\calc/{\text{Tor}}(\calc))$ has one point.  On the other hand, if all elements are of order $2^k$ for some $k$, then by Theorem~\ref{thm:tor2} we have
$\P(\calc/{\text{Tor}}(\calc))$ is in bijective correspondence with $\P(\Z_2^\infty)$.  We summarize these observations with the following theorem.

\begin{theorem} Either $ \P(\calc) = \P(Z_2^\infty)  \sqcup \P(\Q^\infty)$ or $ \P(\calc) = {*}  \sqcup \P(\Q^\infty)$, where $*$ is a single point.  The first case holds if $\calc$ contains no odd order torsion.  The second case holds if there is odd torsion.
\end{theorem}
%%%%%%%Section%%%%%%%%%%%%%%

\section{Metrics on $\P^*(M)$.}\label{sec:metrics}

Suppose that $d$ is an integer-valued metric on the module $M$. We   show that it induces a metric $\Delta$   on $\P^*(M)$.  

\subsection{Definition of  $\Delta$.} Recall that for an element $x \in M$ we denote its equivalence class  by $[x]\in \P^*(M)$.  Also, $0 \in M$ is the unique representative of the class we have denoted $*$. 

\begin{definition}$ $
\begin{enumerate}

\item   For $[x] \in \P^*(M)$ and $ [y] \in \P^*(M)  $,   \[ \delta  ([x],[ y]) = \inf \{ d(x',y') \ \big| \  x' \in [x]\  ,\   y' \in [y]\}.\]

\item  For $[x] \in \P^*(M)$ and $ [y] \in \P^*(M)  $, 
\[
\Delta([ x],[y ]) = \min\{\delta( [x_0],[ x_1 ])+ \delta( [x_1],[ x_2] ) + \cdots + \delta( [x_{n-1}], [x_n ])\},
\]  where the minimum is taken over all sequences of classes for which $[x_0] =[ x]$ and $[x_n] =[ y]$.

\end{enumerate}
\end{definition}

Elementary examples demonstrate the need of considering chains to achieve transitivity.  The proof of the following result is immediate, given that $d$ is integer-valued.

\begin{theorem}  The function $\Delta\co \P^*(M) \times \P^*(M) \to \Z$ is   a metric.
\end{theorem}

\subsection{Mappings of metric spaces $(\P^*(M), \Delta)$.}   

\begin{definition}Suppose that $F\co (X,d) \to (Y,d')$ is a function between metric spaces.    Then $F$ is called a {\it weak contraction} if $d'(F(x_0), F(x_1))  \le d(x_0,x_1)$ for all $x_0$ and $x_1$ in $X$.
\end{definition} 

We have the following elementary result.

\begin{theorem}  If $F \co (M,d) \to (N,d')$ is a weak contraction of modules with integer valued metrics, then $F_*\co (\P^*(M),\Delta)  \to (\P^*(N), \Delta')$ is a weak contraction with respect to the induced metrics.
\end{theorem}

%%%%%%%Section%%%%%%%%%%%%%%

\section{Properties of the metric $\Delta$ on $\P^*(\calc)$}\label{sec:2kc}

Here is the definition of the metric $\Delta$ restated for the special case of knots. 

\begin{definition}$ $
\begin{enumerate}
\item For $\calk$ and $ \calj $ in $ \calc$,   $d(\calk, \calj) = d(K, J) =  g_4(K \cs -J)$, where $K$ and $J$ are arbitrary representatives of $\calk$ and $\calj$.

\item   For $[\calk] \in \P(\calc)$ and $ [\calj] \in \P(\calc)  $,   \[ \delta  ([\calk],[ \calj]) = \min \{ d(\calk', \calj') \ \big| \  \calk' \in [\calk]\  ,\   \calj' \in [\calj]\}.\]

\item  For $[\calk] \in \P(\calc)$ and $ [\calj] \in \P(\calc)  $, 
\[
\Delta([ \calk],[ \calj ]) = \min\{\delta( [\calk_0],[ \calk_1 ])+ \delta( [\calk_1],[ \calk_2] ) + \cdots + \delta( [\calk_{n-1}], [\calk_n ])\},
\] 
where the minimum is taken over all sequences of classes for which $[\calk_0] =[ \calk]$ and $[\calk_n] =[ \calj]$.
\end{enumerate}
\end{definition}
Here is a consequence of Theorem~\ref{thm:infinite} relating the metric $\Delta$ to linear independence in $\calc$.

\begin{theorem}\label{thm:ind}  If $\calk$ and $\calj$ are  elements of infinite order in $\calc$ then $\Delta(\calk, \calj) = 0$ implies that there are   $a, b \in \Z$ such that $a\calk = b\calj \ne 0$.  
\end{theorem}

Example~\ref{ex:6} demonstrates that if  $\calc$ contains an element of odd order,  then the converse does not generalize to the case of knots of finite order.  On the other hand, if $\calc \cong \Z_2^\infty \oplus \Z^\infty$ as might be conjectured, then the condition that $\calk$ and $\calj$ are of infinite order could be dropped.

Another elementary result is the following.

\begin{theorem}\label{thm:distance1}  {\bf (1)} $\Delta(\calk, \calj) = 1$ if and only if $\delta(\calk, \calj) = 1$.  {\bf (2)} If $\delta(\calk, \calj) = 2$ then $\Delta(\calk, \calj) = 2$.
\end{theorem}

Here is one topological result concerning $\Delta([ \calk ], [\calj])$.

\begin{theorem}\label{thm:chain}  If $\Delta([\calk], [\calj]) = n >0 $, then there exists a sequence $\calk = \calk_0 , \calk_1, \ldots , \calk_n = \calj$ such that for $0 \le i \le n-1$,   $\delta(\calk_i , \calk_{i+1}) = 1 $.
\end{theorem}

\begin{proof} We work with $\delta$ and prove the analogous statement. This clearly implies the result for $\Delta$. Thus, assume  $\delta([\calk], [\calj]) = n >0 $.  Then there exists representative knots $K' \in [\calk]$ and $J' \in   [\calj]$ for which $d(K', J') = n$.  

Let $C$ be a genus $n$ cobordism from $K'$ to $J'$.  An isotopy can be performed so that the  maximums   all occur first and the minimum last.  Call the count of these $M$ and $N$.  The saddle points can be put in arbitrary order, the count of these will call $S$. From the Euler characteristic, we know that genus of the cobordism satisfies $n = (S - M -N)/2$. 

By ordering the saddle points, we can arrange the the first $N$ saddle points create a genus 0 cobordism,   a concordance, between $K'$ and a knot $K''$.  That is, after the maximum are passed, there are $N+1$ components, and the first $N$ saddle points reconnect the curve.  Similarly,  the remaining saddle points can be paired so that the last ones along with the minimums form a concordance from $J'$ to a knot $J''$.  

We are now left with a cobordism of genus $n$ from $K''$ to $J''$ containing only saddle points.  Those saddle points can now be ordered to form a set of $(S-M-N)$ pairs:   the first of each pair disconnects the curve, and the second reconnects them.   Thus, we have built a cobordism that consists of a sequence of cobordisms, each of genus 1.  The knots formed in this process constitute the desired knots $K_i$. 
\end{proof}

%%%%%%%Section%%%%%%%%%%%%%%

\section{Computation the projective distance}\label{sec:distance1}
In general, computing the projective distance  $\delta([\calk],[ \calj])$ is inaccessible, and computing $\Delta$ is even more difficult.  For instance, if $\calc$ contains elements that are infinitely divisible, it is hard to imagine what tools could effectively measure the distance between all divisors  for a pair of such knots.  Thus, we will want to restrict ourselves to knots that are primitive, using an additive function to do so, and then use perhaps other additive functions to bound the distance.  

We begin with an elementary observation and move to Theorem~\ref{thm:bound} which provides a tool in our computations of $\delta$.  In the next section we consider the metric $\Delta$.

\begin{theorem} 
Let  $\nu\co \calc \to \Z$ and  $\psi \co \calc \to \R$ be additive functions.  If $\nu(\calk) = 1$, then for all $\calk' \in [\calk]$, $\psi(\calk') = \nu(\calk') \psi( \calk)$.
\end{theorem}

\begin{proof}
If we simplify the condition that $\nu(\calk) = 1$  to $\nu(\calk) \ne 0$, it is  clear from the definition of $\sim'$ that if $\calk \sim' \calk'$,  then $(\nu(\calk'), \psi(\calk')) = c  (\nu(\calk), \psi(\calk))$ for some $c \ne 0 \in \Q$.    Thus, it quickly follows that if  $\calk \sim  \calk'$, we also have  $(\nu(\calk'), \psi(\calk')) = c  (\nu(\calk), \psi(\calk))$ for some $c \ne 0 \in \Q$.  If in this equation apply the condition that   $\nu(\calk) = 1$,  we arrive at the desired result. 
\end{proof}

\begin{theorem}\label{thm:main-bound} 
Let  $\nu_1\co \calc \to \Z$ and  $\nu_2 \co \calc \to \Z$ be additive functions and let $\psi \co \calc \to \R$ be an additive function satisfying $\big| \psi(\calk) \big| \le g_4(\calk)$ for all $\calk \in \calc$.  Suppose that $\nu_1(\calk) = 1$ and $\nu_2(\calj) = 1$.  Then 
\[ 
\delta([\calk],[ \calj]) \ge \min\{ \psi(a\calk \cs -b\calj)\ \big|  \ a, b \ne 0\}.
\]
\end{theorem}

Let $\Omega$ denote a set of   real-valued additive invariants on the concordance group that give lower bounds on the four-genus. Let $\nu_1$ and $\nu_2$ be $\Z$--valued additive invariants.  The following is  now immediate.

\begin{theorem} \label{thm:bound} Suppose that $\nu_1(\calk) = 1$ and $\nu_2(\calj) = 1$.  Then 
\[
\delta([\calk],[\calj]) \ge \min \big\{ |  \max\{ \big |a\psi(\calk) - b\psi(\calj) | \ \big| \  \psi \in \Omega  \}\ 
\big|   \ a \in \Z, b\in \Z, ab \ne 0
\big\}.
\]
\end{theorem}

\begin{example}
We apply Theorem~\ref{thm:main-bound} to show that  $\delta([\calt_{2,3}],[\calt_{2,13}] )= 2$.  

Let $\nu_1(K) =    \sigma'_*(\frac{1}{3} +\epsilon)$ and let  $\nu_2(K) =    \sigma'_*(\frac{1}{13} +\epsilon)$.  Then these satisfy the conditions required by Theorem~\ref{thm:main-bound}.   Our set of homomorphisms $\Omega$ will be the set of signature functions $\sigma'_*(t)$ for $0 \le t \le 1$.  We then have
\[
\delta([\calt_{2,3}],[\calt_{2,13}]) \ge \min \big\{ |  \max\{ \big |b\sigma'_{ T_{2,13}}(t)- a\sigma'_{T_{2,3}}(t)  | \ \big| \  t\in [0,1]  \}\ 
\big|   \ a \in \Z, b\in \Z, ab \ne 0
\big\}.
\]
Considering the signature at $x = \frac{3}{13} + \epsilon$ we have $\sigma'_{T_{2,3}}(x) = 0$ and thus for all $a$,  
\[
\big | b\sigma'_{T_{2,13}} (t)          - a      \sigma'_{T_{2,3}}(t)            \big| \ge 2b.
\]
It follows that  $\delta([\calt_{2,3}],[\calt_{2,13}]) \ge 2$.  A construction such as in Section~\ref{sec:geo2k} shows that $g_4( T_{2,13} \cs -4T_{2,3}) \le 2$.
\end{example}

The proof of Theorem~\ref{thm:growth2} relied on the signature function, so we have the following corollary of Theorem~\ref{thm:main-bound}. 

\begin{corollary}\label{thm:growth3}  For any fixed integer $k \ge 1$, 
 \[ 
 \lim_{n \to \infty}  \frac{\delta([\calt_{2,2k+1}],[\calt_{2n+1}])  }{n} = \frac{1}{2k+1}.
 \]
 \end{corollary}

\subsection{Failure of the triangle inequality for $\delta$.} \label{sec:failure}   The results of Section~\ref{ex:91} gave us  the following. 

\begin{itemize}
\item  $\delta(\{[\calt_{2,41} ], [\calt_{2,61}]) = 2$.
\item  $\delta([\calt_{2,61}] ,[ \calt_{2,91}]) = 2$.
\item  $\delta([\calt_{2,41}] ,[ \calt_{2,91}]) = 5$.
\end{itemize}

To expand on this, w applying Theorem~\ref{thm:distance1} we have

\begin{itemize}
\item  $\Delta([\calt_{2,41} ], [\calt_{2,61}]) = 2$.
\item  $\Delta([\calt_{2,61}] , [\calt_{2,91}]) = 2$.
\item  $\Delta([\calt_{2,41}] ,[ \calt_{2,91}])\le 4$.
\end{itemize}

With this, the necessity of considering chains is defining $\Delta$ is apparent.

%%%%%%%Section%%%%%%%%%%%%%%

%%%%%%%Section%%%%%%%%%%%%%%

\section{The metric $\Delta$ on $\P(\calc)$ and balls of small radius.}\label{sec:small ball}
Recall the definition.

\begin{definition}\label{def;distance}  For $[\calk] \in \P(\calc)$ and $ [\calj] \in \P(\calc)  $, 
\[
\Delta([ \calk],[ \calj ]) = \min\{\delta( [\calk_0],[ \calk_1 ])+ \delta( [\calk_1],[ \calk_2] ) + \cdots + \delta( [\calk_{n-1}], [\calk_n ])\},
\] 
where the minimum is taken over all sequences of classes for which $[\calk_0] =[ \calk]$ and $[\calk_n] =[ \calj]$.
\end{definition}
According the Theorem~\ref{thm:chain}, the minimum can be realized by chains in which each step is of distance 1.  Here we consider balls of radius 1, restricting to the span of $(2, 2k+1)$--torus knots.

\subsection{Balls of radius one}
We begin by  considering  balls of small radius, restricting or examples to the image of the  subspace spanned by the $\calt_{2,2n+1}$.   There is one important observation.  Let $\cals$ denote the span in $\calc$ of classes represented by two-stranded torus knots.  We could have constructed a space $\P(\cals)$ and defined a projective metric $\Delta_{\cals}$.  We are not asserting the the map $(\cals, \Delta_\cals) \to  (\calc, \Delta)$ is an isometric embedding;  see Section~\ref{sec:span}.

\begin{theorem} If $n>k$ and $\delta  ([\calt_{2,2k+1}],   [\calt_{2,2n+1}]) = 1$, then either {\bf (1)} $n = k+1, 2k +1$, or $3k+1$, in which case the minimum is realized by 
$T_{2,2n+1} - \alpha T_{2k+1}$, or   {\bf (2)} $n = 2k$, in which case the minimum is realized by 
$T_{2,2n+1} - (\alpha +1)  T_{2k+1}$. 
\end{theorem}

\begin{proof}
If $\delta  ( [ \calt_{2,2k+1}], [\calt_{2,2n+1}]  ) = 1$, then for some $a$ and $b$,   $g_4 ( bT_{2,2n+1}\cs - aT_{2,2k+1}) = 1$.  The signature condition implies that $a=1$ and we are in the setting of Theorem~\ref{thm:a=1}. 

The genus 1 surface is built as illustrated in Figures~\ref{fig:schematic1}~and~\ref{fig:schematic2}.  In those diagrams, many of the bands have surgery curves going over them.  Let the number of bands on the upper set that do not interact with surgery curves be denoted $U$.  Let the lower count be $L$.   In Figure~\ref{fig:schematic1} we have $U = 1 + 3 = 4$ and $L=0$.     In  Figure~\ref{fig:schematic2} we have $U = 2$ and $L=2$. An important observation is that after surgery, the genus of the surface that results is $(U+L)/2$.   

 By Theorem~\ref{thm:a=1} we need to consider two cases:   $  T_{2,2n+1} \cs - \alpha T_{2,2k+1}  $ and $  T_{2,2n+1} \cs - (\alpha-1) T_{2,2k+1} $ (recall that $\alpha = \lfloor \frac{2n+1}{2k +1}\rfloor$).

\medskip

\noindent{\bf Case 1:  $  T_{2,2n+1}\cs - \alpha T_{2,2k+1}  $.  }  (See Figure~\ref{fig:schematic1}.) In this case we have that $L=0$ and $U = (\alpha -1) +\big( 2n - ( \alpha(2k+1) -1)\big) $ which simplifies to give 
\[ (U + L)/2 = n - \alpha k.\]
If this equals 1, so that $n = \alpha k +1$,  we find 
\[\alpha = \Big \lfloor \frac{2n+1}{2k+1}\Big \rfloor = \Big \lfloor \frac{2\alpha k +2+1}{2k+1}\Big  \rfloor = \Big \lfloor \frac{2\alpha k + \alpha +3 - \alpha}{2k+1}\Big  \rfloor
= \Big \lfloor \alpha +  \frac{ 3 - \alpha}{2k+1}\Big  \rfloor.
\]
Since $\alpha$ is a positive integer, this can  occur only  if $1 \le \alpha \le 3$. 

\medskip

\noindent{\bf Case 2:  $  T_{2,2n+1} \cs - (\alpha +1) T_{2,2k+1}  $.  }   (See Figure~\ref{fig:schematic2}.)  In this case,   $U = \alpha$.  For the lower surface we have $L = 2k - \big(    2n -  \alpha (2k+1)              \big)  =  2k(\alpha +1) + \alpha -2n.$  Thus, we have
\[ (U + L) /2 = k\alpha + k +\alpha - n.
\]
Thus, if this is 1, we have $n = k \alpha +k+\alpha -1$, and 
\[\alpha = \Big \lfloor \frac{2n+1}{2k+1}\Big \rfloor = \Big \lfloor \frac{2\alpha k +2k + 2\alpha - 2 +1 }{2k+1}\Big  \rfloor 
= \Big \lfloor \frac{2\alpha k + \alpha +2k +   \alpha - 1 }{2k+1}\Big  \rfloor
= \Big \lfloor \alpha  + \frac{ 2k +   \alpha - 1 }{2k+1}\Big  \rfloor.
\]
This is possible only if 
\[  \frac{ 2k +   \alpha - 1 }{2k+1}  < 1.
\]
This implies that $\alpha < 2$, that is, $\alpha =1$.  This reduces to the case of $n = 2k$, as desired.
\end{proof}

In the following corollary, we make a small change in notation, working with torus knots $T_{2,N}$ rather than $T_{2,2k+1}$.

\begin{corollary}
The ball of radius one about the class  $[\calt_{2,N}]$ contains the following classes: $[\calt_{2,N+2 }]$, $[\calt_{2,2N-1 }]$, $[\calt_{2,2N+1 }]$ and $[\calt_{2,3N }]$, and no other elements  $[\calt_{2,N'}]$ with  $N' > N$. 
\end{corollary}

\begin{example}  The ball of radius one around the class $[\calt_{2,15}]$ consists of the following set:
\[ B_1[(\calt_{2,15}]) = \{[ \calt_{2,5}], [\calt_{2,7}],[ \calt_{2,29}],[ \calt_{2,31}], [\calt_{2,45}]\}. 
\]
\end{example}

\subsection{Balls of radius two} We study balls of radius two only to the extent needed to build the following example.

\begin{example}\label{ex:2-25} $\Delta([\calt_{2,5} ], [\calt_{2,25}]) = 2$, but for any class $\calj$ for which $\delta([\calt_{2,5}],[ \calj]) = 1$  and $\delta([\calj], [\calt_{2,25}]) =1$, $\calj$ cannot be of the form $\calt_{2,2k+1}$.  The explanation is given next.
\end{example}

Here is a the general result that shows that $\Delta([\calt_{2,5}] , [\calt_{2,25}]) = 2$. 

\begin{lemma}   For $n = 5k+2$,  $\Delta([\calt_{2,2k+1}], [\calt_{2,2n+1}]) = 2$.  
\end{lemma}

\begin{proof}  A construction such as illustrated in Figure~\ref{fig:schematic1} shows that if $n = 5k +2$, then  $\delta([\calt_{2,2k+1}],[ \calt_{2,2n+1}]) \le 2$.  The signature function, evaluated at $\frac{3}{10k + 5} + \epsilon$ can be used to show this is an equality.
\end{proof}

\begin{itemize}
\item  $k= 1, n = 7$.   We have that  $\Delta(\calt_{2,3}, \calt_{2,15}) = 2$.   Notice that there is a chain:  $\delta(\calt_{2,3}, \calt_{2,5}) = 1$ and $\delta(\calt_{2,5}, \calt_{2,15}) = 1$.  There is also the chain:  $\delta(\calt_{2,3}, \calt_{2,7}) = 1$ and $\delta(\calt_{2,7}, \calt_{2,15}) = 1$.

\item $k = 2, n = 12$.  We have that  $\Delta(\calt_{2,5}, \calt_{2,25}) = 2$.  However, there is no chain of length two among two-stranded torus knot classes with both steps of length one.  Starting with $\calt_{2,5}$, the only knots with a $\delta $ distance of 1 are $\calt_{2,3}, \calt_{2,7}$, $\calt_{2,9}$, $\calt_{2,11}$, and $\calt_{2,15}$.  The radius one balls around these include $ \calt_{2,9 }$, $ \calt_{2,13 }$, $ \calt_{2,15 }$, $ \calt_{2,21 }$, $ \calt_{2,11 }$, $ \calt_{2,17 }$,$ \calt_{2,19 }$, $ \calt_{2,27 }$, $ \calt_{2,17 }$, $ \calt_{2,29 }$, $ \calt_{2,31 }$ and  $ \calt_ {2,45 }$.  Notice that $\calt_{2, 25} $ is not on the list.
\end{itemize}
 
Theorem ~\ref{thm:chain} implies that there is a knot $\calj$ for which $\delta(\calt_{2,5}, \calj) = 1$ and $\delta(\calj, \calt_{2,25}) =1$.  One example is $\calj =  2\calt_{2,5} \cs   \calt_{2,15}$.  It is a simple exercise to find a pair of band moves that converts $5\calt_{2,5}$ into $\calj$, and another pair of band moves that converts $\calj$ into $\calt_{2,25}$.

%%%%%%%Section%%%%%%%%%%%%%%

\section{Linear spans}\label{sec:span}

Given any subgroup $\cals \subset \calc$, there is  the projective space $\P(\cals)$ along with the metric $\Delta_\cals$.  We can then consider the metric properties of the map induced by inclusion $\P(\cals) \to \P(\calc)$.  Here are two relevant examples.

\begin{example}\label{ex:4191} Let $\cals =  \left< \calt_{2,41} , \calt_{2,91}\right>$ The inclusion $\P(\cals ) \to \P(\calc)$ is not an isometry.  
From Example~\ref{ex:91} we have that the distance between any two  elements in $\cals$ is at least 5, so that the projective
 distance   $\Delta_\cals([\calt_{2,41}] ,[ \calt_{2,91}])\ge 5$.   In fact, the distance is precisely 5.  On the other hand, we saw in the example that in $\P(\calc)$ we have $\Delta( [\calt_{2,41}] , [\calt_{2,91}]) \le  4$.
\end{example}
 
\begin{example} Let $\cals = \left< \calt_{2,5} , \calt_{2,25}\right>$.  Theorem~\ref{thm:chain} does not hold in $\P  (\cals)$.   We saw in Example~\ref{ex:2-25} that $\Delta( [ \calt_{2,5}] , [\calt_{2,25}]) = 2$.  Since $g_4( 5\calt_{2,5} \cs -\calt_{2,25}) =2$, we have that   $\Delta_\cals([\calt_{2,5}], [\calt_{2,25}]) = 2$.  On the other hand, the only element   $\calj \in  \left< \calt_{2,5} , \calt_{2,25}\right>$ for which $\Delta( [\calt_{2,5}], [\calj]) = 1$ is $\calj = \calt_{2,25} $ itself.
\end{example}

%%%%%%%Section%%%%%%%%%%%%%%

\section{A simplical complex built from $\P(\calc)$}\label{sec:simplicial}

There is a canonical simplicial complex associated to  $(\P(\calc), \Delta)$, denoted $\overline{(\P(\calc), \Delta)}$, which we will abbreviate $\overline{\P(\calc)}$.  We introduce it here to provide concise statements of basic questions about the metric properties of  $ P(\calc)$.

For any set with integer-valued metric, $(X, d)$, there is an embedding of $X$ into a simplicial complex $\overline{(X,d)}$.  By definition, an $n$--simplex 
of  $\overline{(X,d)}$ consists of a set of distinct elements $\{x_0, \ldots, x_n\}$ such that $d(x_i, x_j) = 1$ if $i\ne j$.  This is an example of a {\it Vietoris-Rips} complex~\cite{MR1368659}.

\begin{example}  The following is a $3$--simplices of $\overline{\P(\calc)}$: $\{ [\calt_{2,3}], [\calt_{2,5}] , [\calt_{2,7}] , [\calt_{2,3} \cs \calt_{2,5}] \}$.

\end{example}

\begin{example}  There exists an infinite set of $n$--simplices in $\overline{\P(\calc)}$.    Let $K_n$ be the $n$--twisted double of the unknot  with clasp chosen so that the Seifert form is   
 $  \begin{pmatrix} % or pmatrix or bmatrix or Bmatrix or ...
     1 & 1 \\
      0 & n \\
   \end{pmatrix}.$
   The set $\{K_n\}_{n \ge 1}$ is linearly independent; this follows from the independence of the signature functions, which  have jumps at complex numbers $e^{\theta_n i} $  where $\cos(\theta_n) = 1 - \frac{1}{2n}$.  (These knots  were first used to show that $\calc$ contains an infinite free summand  by Milnor~\cite{MR0242163}.)   
   
   The knot  $K_n$ can be unknotted with a single negative to positive crossing change.  Thus, $K_n \cs -K_m$ can by unknotted with one positive and one negative crossing change.  It follows that $K_n \cs -K_m$ bounds a disk in $B^4$ with two double points of opposite sign.  A simple tubing construction yields an embedded punctured torus, showing that $\Delta( [\calk_n], [\calk_m]) = 1$.  Thus, any set of $n+1$ of these knots yields an $n$--simplex in $\overline{\P(\calc)}$.

\end{example}

\begin{example}  There exists an infinite set of $n$--simplices in $\overline{\P(\calc)}$ spanned by algebraically slice knots.   We work with the same knots $K_n$ but   use the set of knots  $\{ K_n \}$ where  $n$ is restricted to be of the form  $n= -k(k+1)$ for some $k\ge 2$.  Using the results of~\cite{MR900252}, Jiang~\cite{MR620010} proved that the concordance classes of these knots are linearly independent over $\Z$.  The same proof as in the previous example shows that they are all of $\Delta$--distance 1 from each other. 
 \end{example}

%%%%%%%Section%%%%%%%%%%%%%%
 
\section{Problems.} \label{sec:problems}
Here are a few problems.  
\begin{enumerate}
\item Show that $(\P(\calc), \Delta)$ is unbounded.

\item  Show that every element of $(\P(\calc), \Delta)$ has infinite order; that is, the ball of radius one about every element   is infinite.  In~\cite{MR1905691}, Hirasawa and Uchida proved such a statement  for the Gordian complex of the set of knots, where distance is determined by the minimal number of crossing changes required to convert one knot into another; further results were obtained by Baader~\cite{MR2237515}.  The invariants used in those papers do not seem to be applicable in working with $\P(\calc)$.

\item    Under what conditions on $\cals$ is the map $(\P(\cals),\Delta_\cals)  \to (\P(\calc), \Delta)$ isometric?  Note that in Example~\ref{ex:4191} we saw that the inclusion $\P(\left< \calt_{2,41}, \calt_{2,91}\right> )\to \P(\calc)$ is not an isometry.   One might conjecture that if $\cals$ is the span of positive 
(or strongly-quasipositive) knots, then the inclusion is isometric.  (The importance of strongly-quasipositive knots appeared in the work of Rudolph~\cite{MR1193540} and has been extensively studied from the perspective of Heegaard-Floer theory; see, for instance,~\cite{MR2646650}.)

\item  A metric can be defined on   $\P (\calc / {\text{Torsion}} )$ by modifying Definition~\ref{def;distance} so that the path is restricted to non-torsion classes.  Using this metric,  is the injection $\P (\calc / {\text{Torsion}} ) \to \P(\calc)$ isometric?

\item  Let $\cals$ denote the concordance group of topologically slice knots or the subgroup  generated by knots with Alexander polynomial $A_K(t) =1$.  What can be said about $(\P(\cals), \Delta_\cals)$?  In particular, what is the dimension of $\overline{(\P(\cals), \Delta)}$?

\item  Let $d$ be a metric on $\Z \oplus \Z$; for instance, one can could build $d$ using the $L^1$--norm $\big| (a, b) \big| = \max\{ \big| a\big|, \big| b \big| \}$.  Describe the metric space $(\P(\Z \oplus \Z), \Delta)$.  Notice that $\P(\Z \oplus \Z)$ is in  natural  bijective correspondence with the 1--dimensional rational projective line $\Q\P^1$.  What can be said about the simplicial complex $\overline{\P(\Z \oplus \Z), \Delta)}$?

Here are two simple examples that illustrates a property of $ {\P(\Z \oplus \Z), \Delta)}$.  The arrows indicate steps of length 1.
\[
(8,15) \to (8,14) \sim (4,7) \to (4,6) \sim (2,3) \to (2,2) \sim (1,1)
\]
\[
(135, 173)  \to (135, 174) \sim (45, 58) \to (44,58) \sim (22,29) \to (21,28) \sim (3,4) \to (3,3) \sim (1,1) 
\]
The first example, showing that $\Delta( (8,15) ,(1,1)) \le 3$, points to the fact that in general $ \Delta( (x,y) ,(1,1))$ is bounded above by something of the order of  $\max\{ \log_2(x), \log_2(y)\}$.  The second indicates that this bound is probably  a significant overestimate in many cases. \end{enumerate}

%%%%%%%Section%%%%%%%%%%%%%%

%%%%%%%APPENDIX   A%%%%%%%%%%%%%%

%%%%%%%%%%%%%%%%%%%%%%%%%%%%%%%%%%%%%%%%END

%%%%%%%BIBLIOGRAPHY%%%%%%%%%%%%%%

\bibliography{../../BibTexComplete}
\bibliographystyle{plain}	

\end{document}